\documentclass[12pt]{article}
\usepackage{amssymb, amsmath, latexsym, a4}
\usepackage[latin1]{inputenc}
\usepackage{graphicx}
\usepackage[all]{xy}
\usepackage[usenames]{color}

\newenvironment{proof}[1][Proof]{\noindent\textbf{#1.} }{\ \rule{0.5em}{0.5em}}

\newtheorem{De}{Definition}[section]
\newtheorem{Th}[De]{Theorem}
\newtheorem{Pro}[De]{Proposition}
\newtheorem{Le}[De]{Lemma}
\newtheorem{Co}[De]{Corollary}
\newtheorem{Rem}[De]{Remark}
\newtheorem{Ex}[De]{Example}
\newtheorem{note}[De]{Notation}

\newcommand{\im}{{\sf Im}}
\newcommand{\Id}{{\sf id}}
\newcommand{\inc}{{\sf inc}}
\newcommand{\Ker}{{\sf Ker}}

\newbox\pullbackbox
\setbox\pullbackbox=\hbox{\xy 0;<1mm,0mm>: \POS(4,0)\ar@{-} (0,0) \ar@{-} (4,4)

\endxy}

\newdir{>>}{{}*!/3.5pt/:(1,-.2)@^{>}*!/3.5pt/:(1,+.2)@_{>}*!/7pt/:(1,-.2)@^{>}*!/7pt/:(1,+.2)@_{>}}
\newdir{ >>}{{}*!/8pt/@{|}*!/3.5pt/:(1,-.2)@^{>}*!/3.5pt/:(1,+.2)@_{>}}
\newdir{ |>}{{}*!/-3.5pt/@{|}*!/-8pt/:(1,-.2)@^{>}*!/-8pt/:(1,+.2)@_{>}}
\newdir{ >}{{}*!/-8pt/@{>}}
\newdir{>}{{}*:(1,-.2)@^{>}*:(1,+.2)@_{>}}
\newdir{<}{{}*:(1,+.2)@^{<}*:(1,-.2)@_{<}}

\def\id{\operatorname{id}}

\newcommand{\Z}{{\cal{Z}}}

\begin{document}

\centerline{\bf On some properties preserved by the non-abelian tensor product of Hom-Lie algebras}

\bigskip
\centerline{\bf J. M. Casas$^{(1)}$, E. Khmaladze$^{(2)}$ and N. Pacheco Rego$^{(3)}$}

\bigskip

\centerline{$^{(1)}$ Dpto. Matemática Aplicada I, Universidad de Vigo,  E. E. Forestal}
\centerline{Campus Universitario A Xunqueira, 36005 Pontevedra, Spain}
\centerline{ {E-mail address}: jmcasas@uvigo.es}

\bigskip
\centerline{$^{(2)}$ The University of Georgia and A. Razmadze Math. Inst. of Tbilisi State University,}
\centerline{Kostava Str. 77a, 0171 Tbilisi, Georgia}
\centerline{E-mail address:  e.khmal@gmail.com}

\bigskip
\centerline{$^{(3)}$ IPCA, Departamento de Ciências, Campus do IPCA, Lugar do Aldão,}
\centerline{ 4750-810 Vila Frescainha, S. Martinho, Barcelos, Portugal}
\centerline{ {E-mail address}: natarego@gmail.com}
\bigskip

\date{}

\bigskip \bigskip \bigskip

{\bf Abstract:} We study some properties of the non-abelian tensor product of Hom-Lie algebras concerning the preservation of products and quotients,  solvability and nilpotency,  and describe compatibility with the universal central extensions of perfect Hom-Lie algebras.

\bigskip


\bigskip



\section{Introduction}

Hom-Lie algebras were introduced in \cite{HLS} as a generalization of Lie algebras, where the usual Jacobi identity is twisted by a linear self-map, that leads us to the so-called Hom-Jacobi identity. One of the main reasons of the introduction of Hom-Lie algebras was the construction of deformations of the Witt algebra, which is the Lie algebra of derivations on the Laurent polynomial algebra $\mathbb{C}[t, t^{-1}]$.
When the self-map is the identity map, then the Hom-Jacobi identity reduces to the usual Jacobi identity, and the definition of a Lie algebra is recovered.
Thus, it is natural to seek for possible generalizations of known theories from Lie to Hom-Lie algebras. In this context, during the last years many papers appeared dealing with the investigations of Hom-Lie structures. Often in these investigations many non-obvious algebraic identities
need to be verified and they are sufficiently non-trivial and interesting (see e. g. \cite{CaInPa, CaKhPa, CaKhPa2, ChS, Lar, MS, S, Yau1} and related references given there).

As an example of these generalizations, the non-abelian tensor product of Hom-Lie algebra was recently developed in \cite{CaKhPa} with various applications in homology and central extensions of Hom-Lie algebras. This non-abelian tensor product is the Hom-version of the non-abelian tensor product of Lie algebras developed by Ellis \cite{El1} and used in \cite{Gu, InKhLa} for the construction of the non-abelian homology of Lie algebras. These non-abelian homology has been recently extended to Hom-Lie algebras in \cite{CaKhPa2}, where the non-abelian homology of Hom-Lie algebras is constructed and its applications in the cyclic homology of Hom-associative algebras are given.

In this paper we continue investigation done in \cite{CaKhPa} and obtain further properties of the non-abelian tensor product of Hom-Lie algebras.
In particular, the paper is organized as follows. In Section 2, some basic definitions on Hom-Lie algebras and necessary results for the development of the paper are recalled. In Section 3, some standard results on nilpotency and solvability of Lie algebras are extended to the case of Hom-Lie algebras. In Section 4, properties of the non-abelian tensor product of Hom-Lie algebras given in \cite{CaKhPa} are enriched by new facts  (see Proposition \ref{exact-tensor-3}).  We also describe the conditions under which products and quotients are preserved by the non-abelian tensor product (Proposition \ref{product} and Proposition \ref{quatient}). In Section 5, it is shown   that, under some conditions, the   non-abelian tensor product of Hom-Lie algebras  inherits the properties of nilpotency, solvability and Engel (Theorem \ref{bounds}).
In Section 6, conditions for preservation of perfectness are established (Proposition \ref{perfect}) and  compatibility of the non-abelian tensor product with the universal central extensions of two perfect Hom-Lie algebras is described (Theorem \ref{Th 5.5} and Corollary \ref{cor 6.7}).

\

Throughout this paper we fix a ground field $\mathbb{K}$. All vector spaces are considered over $\mathbb{K}$ and linear maps are $\mathbb{K}$-linear maps. We write $\otimes$ (resp. $\wedge$)
for the tensor product $\otimes_\mathbb{K}$ (resp. the exterior product  $\wedge_\mathbb{K}$).

\section{Preliminaries on Hom-Lie algebras}
We start by reviewing some notions and terminology.

\begin{De}\label{def}
A Hom-Lie algebra $(L, \alpha_L)$ is a vector space endowed with a bilinear bracket operation (product) $[-,-]:L\times L\rightarrow L$ and a linear self-map
$\alpha_L:L\rightarrow L$ satisfying
\begin{align*}
&[x,y] = - [y,x], & \text{ (skew-symmetry)}\\
& [\alpha_L(x),[y,z]]+[\alpha_L(z),[x,y]]+[\alpha_L(y),[z,x]]=0  & \text{(Hom-Jacobi identity)}
\end{align*}
for all $x, y, z \in L$.
\end{De}


In this paper we only  deal with (the so called \emph{multiplicative}) Hom-Lie algebras $(L, \alpha_L)$ such that $\alpha_L$ preserves the bracket, i. e.  $\alpha_L[x,y]=[\alpha_L(x), \alpha_L(y)]$ for all $x, y \in L$. Appropriately, $\alpha_L$ will be called an endomorphism.

\emph{ A homomorphism of Hom-Lie algebras} $f:(L,\alpha_L)\to (L',\alpha_{L'})$ is a linear map $f : L \to L'$ preserving the bracket operation, i.e. $f[x,y]=[f(x),f(y)]$ for all $x,y\in L$, and commuting with $\alpha$ endomorphisms, i. e. $f \circ \alpha_L = \alpha_{L'} \circ f$.

\begin{De} \
\begin{enumerate}
\item[i)] The center of a  Hom-Lie algebra  $(L,\alpha_L)$ is the vector subspace of $L$
\[
\Z(L) = \{x \in L \mid [x, y] =0 \ \text{ for all}\ y \in L\}.
\]
\item[ii)] A Hom-Lie algebra $(L, \alpha_L)$ is called abelian if $\Z(L)=L$, equivalently, the bracket operation  is trivial in $L$, that is  $[x,y] = 0$ for all $x,y \in L$.
\end{enumerate}
\end{De}
\begin{Rem}
When $\alpha_L : L \to L$ is a surjective endomorphism, then $(\Z(L),\alpha_{{L\mid}} )$ is an abelian Hom-Lie algebra and a Hom-ideal of $(L, \alpha_L)$.
\end{Rem}

\begin{Ex}\
\begin{enumerate}
\item[a)] 
Any Lie algebra $L$ can be considered as a Hom-Lie algebra $(L,\id_L)$.
\item[b)]   Let $V$ be a vector space and $\alpha_V:V\to V$ be a linear map, then the pair $(V , \alpha_V)$ is called Hom-vector space. A Hom-vector space $(V,\alpha_V)$ can be considered as an abelian Hom-Lie algebra.
\item[c)] Let $L$ be a Lie algebra, $[-,-]$ be the Lie bracket in $L$ and $\alpha:L \to L$ be a Lie algebra endomorphism. Define $[-,-]_{\alpha} : L \otimes L \to L$ by $[x,y]_{\alpha} = \alpha_L[x,y]$, for all $x, y \in L$. Then $(L, \alpha)$ with the product $[-,-]_{\alpha}$ is a Hom-Lie algebra \cite[Theorem 5.3]{Yau1}.
\item[d)] Any (multiplicative) Hom-associative algebra  $(A,\alpha_A)$ \cite{MS} is the Hom-Lie algebra with the induced bracket operation  $[a,b]=ab-ba$, $a,b\in A$.
\end{enumerate}
\end{Ex}

\begin{De}
A  Hom-Lie subalgebra $(H,\alpha_{H})$ of $(L, \alpha_L)$ is a vector subspace $H$ of $L$, which is closed under the bracket, together with $\alpha_{H}:H\to H$ being  the restriction of  $\alpha_{L }$ to $H$. In such a case we write $\alpha_{L\mid}$ for $\alpha_{H}$.

A  Hom-Lie subalgebra $(H,\alpha_{L\mid})$ of $(L, \alpha_L)$ is said to be a  Hom-ideal if $[h, x] \in H$ for any $h \in H$ and $x\in L$.

If $(H,\alpha_{L\mid})$ is a Hom-ideal of $(L, \alpha_L)$, then $(L/H, \overline{\alpha}_L )$, where  $L/H$ is the quotient vector space and  $\overline{\alpha}_L:L/H\to L/H$ is induced by $\alpha_L$,  naturally inherits a structure of Hom-Lie algebra and it is called quotient Hom-Lie algebra.

Let $(H,\alpha_{L\mid})$  and $(K,\alpha_{L\mid})$  be  Hom-ideals of  $(L,\alpha_L)$. The commutator of $(H,\alpha_{L\mid})$  and $(K,\alpha_{L\mid})$, denoted by $([H,K], \alpha_{{L\mid}})$, is the  Hom-Lie subalgebra of $(L,\alpha_L)$ spanned (as a vector space) by all $[h,k]$,  $h \in H$, $k \in K$.  In  particular, if  $K=L$, then $([H,L], \alpha_{{L\mid}})$ is a Hom-ideal of $(L,\alpha_L)$.
\end{De}

 \begin{De} \label{action} \
 \begin{enumerate}
 \item[1)]
 Let $(L,\alpha_L)$ and $(M, \alpha_M)$ be  Hom-Lie algebras. A Hom-action of $(L,\alpha_L)$ on  $(M, \alpha_M)$ is a linear map $L \otimes M \to M,$ $x \otimes m\mapsto {}^xm$, satisfying the following properties:
 \begin{enumerate}
 \item[a)] ${}^{[x,y]} \alpha_M(m) = {}^{\alpha_L(x)}({}^y m) - {}^{\alpha_L(y)} ({}^x m)$,
 \item [b)] ${}^{\alpha_L(x)} [m,m'] = [{}^x m, \alpha_M(m')]+[\alpha_M (m), {}^x m']$,
 \item [c)] $\alpha_M({}^x m) = {}^{\alpha_L(x)} \alpha_M(m)$
 \end{enumerate}
 for all $x, y \in L$ and $m, m' \in M$.

 The Hom-action is said to be trivial if ${}^xm=0$ for all $x\in L$ and $m\in M$.

\item[{2)}] Let $(M,\alpha_M)$ and $(N,\alpha_N)$ be Hom-Lie algebras with mutual Hom-actions on each other.
The Hom-actions are  said to be compatible if
\[
^{(^mn)} m'=[m',^nm] \quad \text{and} \quad ^{(^nm)}n'=[n',^mn]
\]
for all $m,m'\in M$ and $n,n'\in N$.
\end{enumerate}
\end{De}

\begin{Ex}\ \label{example_action}
\begin{enumerate}
\item[1)] Let  $(L, \alpha_{L})$ be a Hom-subalgebra of a  Hom-Lie algebra $(K,\alpha_{K})$ (maybe $K=L$, which also implies that $\alpha_L=\alpha_K$) and $(M,\alpha_{M})$ be a Hom-ideal of $(K,\alpha_{K})$. Then there is a Hom-action of $(L,\alpha_{L})$ on $(M,\alpha_{M})$ given by the structural bracket  in $K$.
\item[2)] Let $(M,\alpha_{M})$ and $(N,\alpha_{N})$ be two Hom-ideals of a Hom-Lie algebra $(L, \alpha_{L})$. Then there are compatible Hom-actions  of $(M,\alpha_{M})$ and $(N,\alpha_{N})$ on each other given by the structural bracket  in $L$.
\end{enumerate}
\end{Ex}

\begin{De}
 Let $(L,\alpha_L)$ be a Hom-Lie algebra  acting on another Hom-Lie algebra $(M,\alpha_M)$. The semi-direct product Hom-Lie algebra $(L,\alpha_L)$ and $(M,\alpha_M)$ is the Hom-Lie algebra $(M \rtimes L, {\alpha_{\rtimes}})$, with the underlying  vector space $M \oplus L$ and the bracket operation
\[
[(m_1,x_1),(m_2,x_2)]= ([m_1,m_2]+{^{\alpha_L(x_1)}}{m_2} - {^{\alpha_L(x_2)}}{m_1},[x_1,x_2]),
\]
together with the endomorphism ${\alpha_{\rtimes}} : M \rtimes L \to M \rtimes L$ given by ${\alpha_{\rtimes}} (m, x) = \left(\alpha_M(m), \alpha_L(x)\right)$ for all $x,x_1,x_2 \in L$ and $m, m_1, m_2 \in M$.
\end{De}

\section{On solvability  and nilpotency of  Hom-Lie algebras}

Let $(M,\alpha_M)$ be a Hom-ideal of a Hom-Lie algebra $(Q, \alpha_Q)$. \emph{A series } from $(M, \alpha_M)$ to $(Q, \alpha_Q)$ is a finite sequence of Hom-ideals $(M_i, \alpha_{Q\mid})$, $0\leq i \leq k$,  of $(Q, \alpha_Q)$ such that
\[
M = M_0 \unlhd M_1 \unlhd \dots \unlhd M_{k-1} \unlhd M_k = Q.
\]
$k$ is called \emph{the length of the series}.

\;

A series from $(M, \alpha_M)$ to $(Q, \alpha_Q)$ is said to be \emph{central} (resp. \emph{abelian}) if $[Q,M_i]  \subseteq M_{i-1}$, equivalently $M_i/M_{i-1} \subseteq \Z(Q/M_{i-1})$ (resp. if $[M_i,M_i] \subseteq M_{i-1}$, equivalently $[M_i/M_{i-1}, M_i/M_{i-1}] = 0$. )

\;

A series from $(0, \alpha_{Q \mid})$ to $(Q, \alpha_Q)$ is called \emph{a series of the Hom-Lie algebra} $(Q, \alpha_Q)$.

\begin{De}\
\begin{enumerate}
\item[i)] A Hom-Lie algebra $(Q,\alpha_Q)$ is said to be solvable if it has an abelian series. Let $k$ be the minimal length of such series, then $k$ is said to be the class of solvability of $(Q,\alpha_Q)$.
\item[ii)] A Hom-Lie algebra $(Q,\alpha_Q)$ is said to be nilpotent if it has a central series. Let $k$ be the minimal length of such series, then $k$ is said to be the class of nilpotency of $(Q,\alpha_Q)$.
\end{enumerate}
\end{De}

Given a Hom-Lie algebra $(Q, \alpha_{Q})$, consider the sequence of Hom-Lie subalgebras $(Q^{(i)}, \alpha_{Q \mid})$, $i\geq 0$, defined inductively by
\[
Q^{(0)} = Q \quad \text{and} \quad Q^{(i+1)}=[Q^{(i)}, Q^{(i)}], \quad i \geq 0.
\]
It is clear that
\[
\dots \subseteq Q^{(i)} \subseteq \dots \subseteq Q^{(1)} \subseteq Q^{(0)}
\]
and in the particular case of a Lie algebra $Q$ (= $(Q,\id_Q)$), this sequence is exactly the derived series of Lie algebra $Q$.
In general, $(Q^{(i)}, \alpha_{Q \mid})$ terms are Hom-Lie subalgebras and not Hom-ideals of $(Q, \alpha_{Q})$.  This fact is illustrated by the following example.
\begin{Ex}
Consider the four-dimensional Hom-Lie algebra $(Q, \alpha_Q)$ with basis $\{a_1,a_2,a_3,a_4\}$, bracket operation given by $[a_1,a_2]=-[a_2,a_1]=a_1$, $[a_1,a_3]=-[a_3,a_1]=a_2$ and the map $\alpha_Q =0$. It is easy to see that  $Q^{(2)} = \langle \{a_1 \} \rangle$, which is not a Hom-ideal of $(Q,\alpha_Q)$. \end{Ex}

By this reason we can not speak on derived series of a Hom-Lie algebra in the general set-up. Below we indicate wide class of Hom-Lie algebras $(Q, \alpha_Q)$ for which $(Q^{(i)}, \alpha_{Q \mid})$ terms are Hom-ideals for all $i\geq 0$.

\begin{Rem} \label{conditions}
In order to guarantee that the terms $(Q^{(i)}, \alpha_{Q \mid})$, $i \geq 1$, are Hom-ideals, the Hom-Lie algebra $(Q,\alpha_Q)$ should satisfy additional condition. For instance, this is the case  if the endomorphism $\alpha_Q$ is surjective.

Another broad class of Hom-Lie algebras with the same property, are Hom-Lie algebras satisfying the so-called $\alpha$-identity condition given in \cite{CaKhPa}, that is, Hom-Lie algebra $(Q,\alpha_Q)$ such that $[Q, \im(\alpha_Q- \Id_Q)]=0$, which is equivalent to the condition $[x,y]=[\alpha_Q(x),y]$, for all $x, y \in Q$. Examples of this kind of Hom-Lie algebras can be found in \cite{CaKhPa}.

Moreover, actually a weaker condition is sufficient. Namely, it is only needed the condition $[[Q,Q], \im(\alpha_Q- \Id_Q)]=0$, equivalently $[x, [y, z]]=[\alpha_Q(x),[y,z]]$, for all $x, y, z \in Q$. Later we call this condition weak $\alpha$-identity condition.

The weak $\alpha$-identity condition leads us to a broader class of examples. For instance, let $(Q,\alpha_Q)$ be the two-dimensional Hom-Lie algebra with basis $\{a_1,a_2\}$, bracket operation given by $[a_1,a_2]=-[a_2,a_1]= a_1$ and endomorphism $\alpha_Q$ represented by the matrix $\left( \begin{array}{cc} 0 & 1 \\ 0 & 1 \end{array} \right)$ (see \cite{CaInPa}). Then $(Q,\alpha_Q)$ satisfies weak $\alpha$-identity condition, it does not satisfy $\alpha$-identity condition and at the same time $\alpha_Q $ is not surjective.
\end{Rem}

\begin{Th} \label{solvable1} \
\begin{enumerate}
\item[a)] Let $(Q,\alpha_Q)$ be a Hom-Lie algebra and $M = M_0 \unlhd M_1 \unlhd \dots \unlhd M_{j-1} \unlhd M_j = Q$ be an abelian series
from $(M, \alpha_{Q\mid})$ to $(Q, \alpha_Q)$, then $Q^{(i)} \subseteq M_{j-i}$, $0 \leq i \leq j$.
\item[b)] Let $(Q,\alpha_Q)$ be a Hom-Lie algebra such that $\alpha_Q$ is surjective or $(Q,\alpha_Q)$ satisfies the weak $\alpha$-identity condition (see Remark \ref{conditions}). Then $(Q,\alpha_Q)$ is solvable with class of solvability $k$ if and only if $Q^{(k)} =0$ and $Q^{(k-1)} \neq 0$.
\end{enumerate}
\end{Th}
\begin{proof}
{\it a)} This easily follows by induction on $i$.

\noindent{\it b)} If $Q^{(k)} =0$ and $Q^{(k-1)} \neq 0$ then $0= Q^{(k)} \unlhd Q^{(k-1)} \unlhd \dots \unlhd Q^{(1)} \unlhd Q^{(0)} = Q$ is an abelian series of $(Q,\alpha_Q)$ and its length is minimal by {\it a)}. Therefore $(Q,\alpha_Q)$ is solvable of class $k$. The converse is a direct consequence of {\it a)}.
\end{proof}


\

The \emph{lower central series} of a Hom-Lie algebra $(Q, \alpha_{Q})$ is the sequence of Hom-ideals $(Q^{[i]}, \alpha_{Q\mid})$, $i \geq 0$,
 defined inductively by
\[
Q^{[0]} = Q \quad \text{and} \quad Q^{[i+1]}=[Q^{[i]}, Q], \quad i \geq 0.
\]
It is readily checked that each $(Q^{[i]}, \alpha_{Q\mid})$ is indeed a Hom-ideal of $(Q, \alpha_{Q})$ and
\[
\dots \subseteq Q^{[i]} \subseteq \dots \subseteq Q^{[1]} \subseteq Q^{[0]}.
\]

\begin{Th}\ \label{nilpotent}
\begin{enumerate}
\item[a)]  Let $(Q,\alpha_Q)$ be a Hom-Lie algebra and  $M = M_0 \unlhd M_1 \unlhd \dots \unlhd M_{j-1} \unlhd M_j = Q$ be a central series from $(M, \alpha_{Q\mid})$ to $(Q, \alpha_Q)$, then $Q^{[i]} \subseteq M_{j-i}$, $0 \leq i \leq j$.
\item[b)] $(Q,\alpha_Q)$ is nilpotent Hom-Lie algebra with class of nilpotency  $k$ if and only if $Q^{[k]} =0$ and $Q^{[k-1]} \neq 0$.
\end{enumerate}
\end{Th}
\begin{proof}
This is similar to the proof of Theorem \ref{solvable1}.
\end{proof}

\begin{Ex}\
\begin{enumerate}
\item[a)] Abelian Hom-Lie algebras are  nilpotent, hence solvable  Hom-Lie algebras of class 1 (keep in mind that the weak $\alpha$-identity condition is satisfied).

\item[b)] The 2-dimensional complex Hom-Lie algebra \cite{CaInPa} with basis $\{a_1,a_2\}$, bracket operation given by $[a_1,a_2] = - [a_2, a_1]= a_1$ and endomorphism $\alpha$ represented by the matrix $\left( \begin{array}{cc} 0 & {1} \\0 & 1 \end{array} \right)$ (see Remark \ref{conditions}) is non nilpotent, but is solvable of class 2.

\item[c)] The 3-dimensional Hom-Lie algebra  with basis $\{a_1,a_2,a_3 \}$, bracket operation given by $[a_1,a_2]= - [a_2,a_1] = a_3$ and endomorphism $\alpha = 0$  is a nilpotent Hom-Lie algebra of class 2.
\end{enumerate}
\end{Ex}


\section{Non-abelian tensor product of Hom-Lie algebras}

In this section we recall from \cite{CaKhPa} the non-abelian tensor product of Hom-Lie algebras and establish new properties needed in the sequel.
\begin{De}
Let $(M,\alpha_M)$ and $(N,\alpha_N)$ be Hom-Lie algebras with Hom-actions on each other and let $(L,\alpha_L)$ be one more Hom-Lie algebra. A bilinear map $h : M\times N\to L$ is said to be a Hom-Lie pairing if the following properties are satisfied:
\begin{enumerate}
\item [i)] $h([m,m'],\alpha_N(n)) = h(\alpha_M(m), {}^{m'} n) - h(\alpha_M(m'), {}^m n)$,
 \item [ii)] $h(\alpha_M(m),[n,n']) = h({}^{n'} m, \alpha_N(n))  - h({}^n m, \alpha_N(n'))$,
 \item [iii)] $h({}^n  m, {}^{m'} n') = - [h(m,n), h(m',n')]$,
 \item [iv)] $h  \circ (\alpha_M \times \alpha_N) = \alpha_L  \circ h$,
\end{enumerate}
for all $m, m' \in M$, $n, n' \in N$.
\end{De}

\begin{De} \label{Def}
Let $(M,\alpha_M)$ and $(N,\alpha_N)$ be Hom-Lie algebras with compatible mutual Hom-actions on each other.  The non-abelian tensor product of $(M,\alpha_M)$ and $(N,\alpha_N)$ is the Hom-Lie algebra, denoted by $(M,\alpha_M)\star (N,\alpha_N)$ or shorter by $(M\star N, \alpha_{M\star N})$, defined as follows: as vector space, $M\star N$ is the quotient of the tensor product $M\otimes N$ (of the underlying vector spaces) by the vector subspace spanned by all elements of the form
\begin{enumerate}
\item[{\it i)}] $[m,m']\otimes \alpha_N(n) -\alpha_M(m)\otimes {}^{m'}n +\alpha_M(m') \otimes {}^mn$,
\item[{\it ii)}] $\alpha_M(m)\otimes [n,n'] - {}^{n'}m\otimes\alpha_N(n)+{}^{n}m\otimes\alpha_N(n')$,
\item[{\it iii)}] ${}^{n}m\otimes {}^{m}n$,
\item[{\it iv)}]${}^{n}m\otimes {}^{m'}n' +{}^{n'}m'\otimes {}^{m}n$,
\item[{\it v)}]$[{}^{n}m,{}^{n'}m']\otimes \alpha_N({}^{m''}n'')+[{}^{n'}m',{}^{n''}m'']\otimes \alpha_N({}^{m}n)
+[{}^{n''}m'',{}^{n}m]\otimes \alpha_N({}^{m'}n')$,
\end{enumerate}
for $m,m',m''\in M$ and $n,n',n''\in N$.  The bracket operation is defined on generators by
\[
[m\otimes n, m'\otimes n']=-{}^{n}m\otimes {}^{m'}n';
\]
and the linear map $\alpha_{M\star N}$ is naturally induced by $\alpha_M$ and $\alpha_N$.
The equivalence class of $m\otimes n$ will be denoted by $m\star n$.
\end{De}

Below we give some fundamental properties of the non-abelian tensor product of Hom-Lie algebras. First we present its kind of universality property needed in the sequel.

\begin{Le}\label {Lie pairing}
Let $(M,\alpha_M)$ and $(N,\alpha_N)$ be Hom-Lie algebras with compatible Hom-actions on each other. For a Hom-Lie algebra $(L,\alpha_L)$ and any Hom-Lie pairing $h : (M\times N,\alpha_{M \times N}) \to (L,\alpha_L)$ there exists a unique homomorphism $h^{\ast} : (M\star N,\alpha_{M\star N}) \to (L,\alpha_L)$ such that $h^{\ast} (m\star n)=h(m,n)$ for all $m\in M$ and $n\in N$.
\end{Le}
\begin{proof}
It follows directly from Definition 3.8 and Proposition 3.9 in \cite{CaKhPa}.
\end{proof}
\medskip

\begin{Pro} \label{action-on-tensor} \cite{CaKhPa}
Let  $(M,\alpha_M)$ and $(N,\alpha_N)$ be Hom-Lie algebras with compatible Hom-actions on each other.
 \begin{enumerate}
\item[a)] There are homomorphisms of Hom-Lie algebras

$\begin{array}{rl}
&\psi_M:(M\star N, \alpha_{M \star N}) \to (M, \alpha_{M }), \quad \psi_M(m\star n)= -{}^nm,\\
&\psi_N:(M\star N, \alpha_{M \star N}) \to (N, \alpha_N),\quad  \psi_N(m\star n)= {}^mn.
\end{array}$
\item[b)]
There are Hom-actions of  $(M, \alpha_{M })$ and $(N, \alpha_{N })$ on  the Hom-Lie tensor product ($M\star N, \alpha_{M \star N}$) given, for all $m,m'\in M$, $n,n'\in N$, by
\[
\begin{array}{rl}
&{}^{m'}(m\star n)=[m',m]\star \alpha_N(n)+\alpha_M(m)\star {}^{m'}n, \\
 &{}^{n'}(m\star n)={}^{n'}m\star \alpha_N(n)+\alpha_M(m)\star [n',n]
\end{array}
\]
\item[c)] Both $\Ker(\psi_M)$ and $\Ker(\psi_N)$ are contained in the center of $(M\star N, \alpha_{M \star N})$.
\item[d)] The induced Hom-actions of~ $\im(\psi_M)$ on $\Ker(\psi_M)$ and of $\im(\psi_N)$ on $\Ker(\psi_N)$ are trivial.
\item[e)]  $\psi_M$ and $\psi_N$ satisfy the following properties for all $m, m' \in M$, $n, n' \in N$:
\begin{enumerate}
\item[i)] $\psi_M(^{m'}(m \star n)) = [\alpha_M(m'), \psi_M(m \star n)]$,
\item[ii)] $\psi_N(^{n'}(m \star n)) = [\alpha_N(n'), \psi_N(m \star n)]$,
\item[iii)] ${^{\psi_M(m \star n)}}(m' \star n') = [\alpha_{M \star N}(m \star n), m' \star n'] = {^{\psi_N(m \star n)}} (m' \star n')$.
\end{enumerate}
\end{enumerate}
\end{Pro}

\begin {Co} \label{Corollary 4.3} If $t\in M\star N,m'\in M, n' \in N$ then $\psi_M(t)\star \alpha_N(n')= -^{n'}t$ and $\alpha(m')\star \psi_N(t) = {^{m'}}t.$
\end {Co}
\begin{proof}
By statements a) and b) in Proposition \ref{action-on-tensor}, it suffices to prove formulas on generators. By using again Proposition \ref{action-on-tensor} we have
\[
\begin{array}{rl}
\psi_M(t)\star\alpha_N(n')= & \psi_M(m \star n)\star \alpha_N(n')=- {^nm} \star \alpha_N(n')\\
 = & -(\alpha_M(m)\star [n',n]+{^{n'}m} \star \alpha_N(n))=-{^{n'}}(m\star n)=- {^{n'}t},
 \end{array}
 \]
and
\[
\begin{array}{rl}
\alpha(m')\star \psi_N(t)= & \alpha(m')\star \psi_N(m \star n)=\alpha(m')\star ^mn\\
= & [m',m]\star \alpha(n)+\alpha(m)\star {^{m'}n}={^{m'}(m\star n)}= {^{m'}t}.
\end{array}
\]
\end{proof}

The non-abelian tensor product of Hom-Lie algebras is symmetric by means of the following isomorphism
\[
(M\star N, \alpha_{M\star N})\overset{\approx}{\longrightarrow} (N\star M, \alpha_{N\star M}), \quad m\star n\mapsto n\star m.
\]
Furthermore, it is functorial in the sense of the following proposition.
\begin{Pro}\label{composition} \cite{CaKhPa}
If $f:(M,\alpha_M)\to (M',\alpha_{M'})$ and $g:(N,\alpha_N)\to (N',\alpha_{N'})$ are homomorphisms of Hom-Lie algebras together with compatible Hom-actions of $(M, \alpha_M)$ (resp. $(M', \alpha_{M'})$ ) and $(N, \alpha_N)$ (resp. $(N', \alpha_{N'})$) on each other such that $f$, $g$ preserve these Hom-actions, that is
\[
f({}^nm)={}^{g(n)}f(m), \quad g({}^mn)={}^{f(m)}g(n), \qquad m\in M,\ n\in N.
\]
Then there is a unique homomorphism of Hom-Lie algebras
\[
f\star g:(M\star N,\alpha_{M\star N})\to (M'\star N',\alpha_{M'\star N'}), \quad \text{given by} \quad (f\star g)(m\star n)=f(m)\star g(n).
\]
Furthermore, if $f$, $g$ are onto then so is $f \star g$.
\end{Pro}

A sort of right exactness property of the non-abelian tensor product of Lie algebras is presented in the following proposition.

\begin{Pro}\label{exact-tensor-1} \cite{CaKhPa}
Let $0\to (M_1,\alpha_{M_1})\overset{f}{\to}(M_2,\alpha_{M_2})\overset{g}{\to}(M_3,\alpha_{M_3})\to 0$ be a short exact sequence of Hom-Lie algebras. Let $(N,\alpha_{N})$ be a Hom-Lie algebra together with compatible Hom-actions of $(N,\alpha_{N})$ and $(M_i,\alpha_{M_i})$ $(i=1,2,3)$ on each other and $f$, $g$ preserve these Hom-actions.
Then there is an exact sequence of Hom-Lie algebras
\[
(M_1\star N,\alpha_{M_1\star N})\overset{f\star \id_N}{\longrightarrow}(M_2\star N,\alpha_{M_2\star N})\overset{g\star \id_N}{\longrightarrow}(M_3\star N,\alpha_{M_3\star N})\longrightarrow 0.
\]
\end{Pro}

Now we give a more general result needed in the sequel.

\begin{Pro}\label{exact-tensor-3}
Let
\[0\longrightarrow (M,\alpha_{M})\overset{i_1}{\longrightarrow}(L,\alpha_{L})\overset{\sigma_1}{\longrightarrow}(P,\alpha_{P})\longrightarrow 0\]
 \[0\longrightarrow (N,\alpha_{N})\overset{i_2}{\longrightarrow}(K,\alpha_{K})\overset{\sigma_2}{\longrightarrow}(Q,\alpha_{Q})\longrightarrow 0\]
be two short exact sequences of Hom-Lie algebras. Suppose that $\alpha_L$ and $\alpha_K$ are surjective endomorphisms,   $(L,\alpha_{L})$ and $(K,\alpha_{K})$, as well as $(P,\alpha_{P})$ and $(Q,\alpha_{Q})$ act compatibly on each other, and all $\sigma_1$, $\sigma_2$, $\alpha_L$ and $\alpha_K$   preserve the corresponding Hom-actions. Then there is an exact sequence of Hom-Lie algebras
 \[(M \star K,\alpha_{M \star K})\rtimes (L \star N,\alpha_{L \star N})\overset{\eta}{\longrightarrow}(L \star K,\alpha_{L \star K})\overset{\sigma_1 \star \sigma_2}{\longrightarrow}(P \star Q,\alpha_{P \star Q})\longrightarrow 0 \  .\]
\end{Pro}
\begin{proof}
The Hom-action of $(L,\alpha_{L})$ on $(K,\alpha_{K})$ restricts to a Hom-action of  $(L,\alpha_{L})$ on $(N,\alpha_{N})$, and also to a Hom-action of $(M,\alpha_{M})$ on $(K,\alpha_{K})$. Analogously,
the Hom-action of $(K,\alpha_{K})$ on $(L,\alpha_{L})$ restricts to a Hom-action of  $(K,\alpha_{K})$ on $(M,\alpha_{M})$, and also to a Hom-action of $(N,\alpha_{N})$ on $(L,\alpha_{L})$. Since  actions are compatible, then the non-abelian tensor products $(M \star K,\alpha_{M \star K})$ and $(L \star N,\alpha_{L \star N})$ can be considered. Furthermore, there is a Hom-action of $(L \star N,\alpha_{L \star N})$ on $(M \star K,\alpha_{M \star K})$  given on generators by
\[
^{(l \star n)}(m \star k) = - ^nl \star ^mk,
\]
for any $m \in M$, $k \in K$, $l \in L$, $n \in N$. Via this Hom-action the semi-direct product
$(M \star K,\alpha_{M \star K})\rtimes (L \star N,\alpha_{L \star N})$ is defined.

Thanks to Proposition \ref{composition}, there is a homomorphism of Hom-Lie algebras $\sigma_1 \star \sigma_2 : (L \star K,\alpha_{L \star K}) {\to}(P \star Q,\alpha_{P \star Q})$, which is clearly a surjection.
Let $\eta_1 = i_1 \star \Id : (M \star K,\alpha_{M \star K}) \to (L \star K,\alpha_{L \star K})$ and $\eta_2 = \Id \star i_2 : (L \star N,\alpha_{L \star N}) \to (L \star K,\alpha_{L \star K})$ be the homomorphisms provided again by Proposition \ref{composition}. Define the homomorphism of Hom-Lie algebras
\[
 \eta : \left( M \star K,\alpha_{M \star K})\rtimes (L \star N,\alpha_{L \star N} \right)  {\to} ( L \star K,\alpha_{L \star K})
 \]
by $\eta (m \star k,  l \star n) =  \eta_1(m \star k) + \alpha_{L \star K} \circ \eta_2(l \star n)$.

Now it remains to show that $\im(\eta) = \Ker(\sigma_1 \star \sigma_2)$. Obviously $(\sigma_1 \star \sigma_2) \circ \eta = 0$ and hence we have $\im(\eta)\subseteq \Ker(\sigma_1 \star \sigma_2)$. We claim that $\im(\eta)$ is a Hom-ideal of  $(L \star K,\alpha_{L \star K})$. Indeed, since $\alpha_L$ and $\alpha_K$ are surjective homomorphisms, for any $l_1\in L$ and $k_1\in K$, there exist $l'_1\in L$ and $k'_1\in K$ such that $\alpha_L(l'_1)=l_1$ and $\alpha_K(k'_1)=k_1$.  Then, just by using defining relations of the non-abelian tensor product, we get
\[
\begin{array}{lcl}
[l_1\star k_1, \eta \left( m \star k,  l \star n \right)] &=& [l_1 \star k_1, m \star k] + [l_1 \star k_1, \alpha_L(l) \star \alpha_N(n)]  \\
 & = & -{^{k_1}{l_1}}\star {^m}k - {^{k_1}{l_1}}\star {^{\alpha_L(l)}}\alpha_N(n) \\
  & = & -{^{\alpha(k'_1)}}{\alpha_L(l'_1)} \star {^m}k - {^{\alpha(k'_1)}}{\alpha_L(l'_1)} \star {^{\alpha_L(l)}}\alpha_N(n) \\
 & = & - [^{k'_1}l'_1, m] \star \alpha_K(k) - \alpha_M(m) \star ^{(^{k'_1} l'_1)}k  \\
  &  & - [^{k'_1}l'_1, \alpha_L(l)] \star \alpha^2_N(n) - \alpha^2_L(l) \star ^{(^{k'_1}l'_1)}\alpha(n)\\
& = & ^{(^{l'_1}k'_1)}m \star \alpha_K(k)  - \alpha_M(m) \star ^{(^{k'_1}l'_1)} k  \\
&  &  + ^{(^{l'_1}k'_1)} \alpha(l) \star \alpha^2_N(n)  - \alpha^2_L(l) \star  ^{(^{k'_1}l'_1)} \alpha(n) \in \im(\eta).
 \end{array}
 \]
Thus, we get a natural surjective homomorphism \[
\theta:(L \star K, \alpha_{L \star K})/\im(\eta) \longrightarrow (P \star Q, \alpha_{P \star Q}).
\]

On the other hand, the map $h:(P \times Q, \alpha_{P\times Q})\to (L \star K, \alpha_{L \star K})/\im(\eta)$ given by $h(\sigma_1(l), \sigma_2(k))=l \star k + \im(\eta)$ is a well-defined Hom-Lie pairing.
 Thus, by Lemma \ref{Lie pairing}, $h$ induces a homomorphism  ${\theta'}:(P \star Q,\alpha_{P \star Q})\rightarrow (L \star K, \alpha_{L \star K})/\im(\eta)$, $\theta'(\sigma_1(l), \sigma_2(k))=l \star k + \im(\eta)$, which is the two sided inverse to $\theta$. Thus, $\theta$ is an isomorphism. This completes the proof.
 \end{proof}

\begin{Rem}
When all $\alpha$ endomorphisms in Proposition \ref{exact-tensor-3} are identity maps, i. e. all Hom-Lie algebras are Lie algebras, then \cite[Proposition 9]{El1} is recovered.
\end{Rem}

\begin{note}
Let $(M, \alpha_M)$ and $(N, \alpha_N)$ be two Hom-Lie algebras acting compatibly on each other. We denote by ${^NM}$ (respectively, ${^MN}$) the subspace of $M$ (resp. $N$) generated by the elements of the form ${^{n}m}$ (resp. ${^{m}n}$), with $m\in M$ and $n \in N$.
 It follows from the compatibility conditions that $(^NM,\alpha_{M \mid })$ and $(^MN,\alpha_{N \mid})$ are Hom-ideals of $(M, \alpha_M)$ and $(N, \alpha_N)$, respectively.

  Note that when $M=N$ and the Hom-actions are given by the structural bracket operation as in Example \ref{example_action}, then we have
$\left( {^MM},\alpha_{M\mid} \right) = \left( [M,M],\alpha_{M\mid} \right)$.
\end{note}

\begin{Pro} \label{product}
Let $\left( A,\alpha _{A}\right)$, $\left( B,\alpha_{B}\right)$ and $\left( C,\alpha _{C}\right)$ be Hom-Lie algebras and the following three conditions are satisfied:
\begin{enumerate}
\item[1)] $\left( A,\alpha _{A}\right)$ and $\left( B,\alpha _{B}\right)$ act compatibly on each other, and  $\left( A,\alpha _{A}\right)$ and $\left( C,\alpha _{C}\right)$ act compatibly on each other as well;

\item[2)] $^{\alpha _{B}\left( b\right) }\left( ^{c}a\right) = ~^{\alpha_{C}\left( c\right) }\left( ^{b}a\right)$, for all $ a \in A$, $b \in B$, $c \in C$;

\item[3)] The canonical maps $(^{C}A,\alpha_{A\mid}) \star (\alpha_B(B), \alpha_{B\mid}) \longrightarrow  (A, \alpha_A) \star (\alpha_B(B),\alpha_{B\mid})$ and $(^{B}A,\alpha_{A\mid}) \star (\alpha_C(C),  \alpha_{ C\mid}) \longrightarrow  (A,\alpha_A) \star (\alpha_C(C), \alpha_{C\mid })$ are trivial, and the induced Hom-actions of $^BA $ on $C$ and of $^CA$ on $B$ are also trivial.

\end{enumerate}
 Then there is an  isomorphism of Hom-Lie algebras
 \[
 \left( A,\alpha _{A}\right) \star \left( \left( B,\alpha _{B}\right) \times \left( C,\alpha _{C}\right)\right) \cong \left( \left( A,\alpha _{A}\right) \star \left( B,\alpha _{B}\right) \right) \times \left( \left( A,\alpha _{A}\right) \star \left( C,\alpha _{C}\right)\right).
 \]
\end{Pro}
\begin{proof}
First let us remark  that $ \left( B,\alpha _{B}\right) \times \left( C,\alpha _{C}\right)$ denotes the Hom-Lie algebra $\left( B\times C,\alpha _{B\times C}\right)$, with the endomorphism $\alpha _{B\times C}:B\times C\to B\times C $ naturally induced by $\alpha_B$ and $\alpha_C$.
There are mutual compatible Hom-actions of $\left( B\times C,\alpha _{B\times C}\right)$ and  $\left( A,\alpha _{A}\right)$ on each other given for all $a \in A$, $b\in B$, $c \in C$ by $^{\left( b,c\right)}a= ~{^{b}a} + {^{c}a}$ and $^a(b,c) = ({^ab}, {^ac})$, respectively. Hence  $\left( A,\alpha _{A}\right) \star \left( \left( B,\alpha _{B}\right) \times \left( C,\alpha _{C}\right)\right)$ can be considered.

Now we define the map
\[
\phi : \left( A,\alpha _{A}\right) \times \left( B \times  C, \alpha _{B \times C}\right) \to \left( \left( A,\alpha _{A}\right) \star \left( B,\alpha _{B}\right) \right) \times \left( \left( A,\alpha _{A}\right) \star \left( C,\alpha _{C}\right)\right)
 \]
 by $\phi(a, (b,c)) = (a \star b, a \star c)$. Keeping in mind the conditions {\it 2)} and {\it 3)},  it is straightforward to verify that $\phi$ is a Hom-Lie pairing, so Lemma \ref{Lie pairing} provides a unique homomorphism
\[
\phi^{\star} : \left( A,\alpha _{A}\right) \star \left( B \times  C, \alpha _{B \times C}\right) \to \left( \left( A,\alpha _{A}\right) \star \left( B,\alpha _{B}\right) \right) \times \left( \left( A,\alpha _{A}\right) \star \left( C,\alpha _{C}\right)\right)
\]
 given by $\phi^{\star}(a \star (b,c)) = (a \star b, a \star c)$.

On the other hand, the maps
\[
\psi_1 : \left( A,\alpha _{A}\right) \times \left( B,\alpha _{B}\right) \to \left( A,\alpha _{A}\right) \star \left( B \times  C, \alpha _{B \times C}\right),
\quad \psi_1(a,b) = a \star (b,0),
\]
 and
\[
\psi_2 : \left( A,\alpha _{A}\right) \times \left( C,\alpha _{C}\right) \to \left( A,\alpha _{A}\right) \star \left( B \times  C, \alpha _{B \times C}\right), \quad
\psi_2(a,c) = a \star (0,c)
\]
 are Hom-Lie pairings, so again  thanks to Lemma \ref{Lie pairing} they induce homomorphisms $\psi_1^{\star} : \left( A,\alpha _{A}\right) \star \left( B,\alpha _{B}\right) \to \left( A,\alpha _{A}\right) \star \left( B \times  C, \alpha _{B \times C}\right)$, $\psi_1^{\star}(a \star b) = a \star (b,0)$, and
$\psi_2^{\star} : \left( A,\alpha _{A}\right) \star \left( C,\alpha _{C}\right) \to \left( A,\alpha _{A}\right) \star \left( B \times  C, \alpha _{B \times C}\right)$, $\psi_2^{\star}(a \star c) = a \star (0,c)$. Both homomorphisms induce the homomorphism
\[\psi^{\star} = ( \psi_1^{\star}, \psi_2^{\star} ) : \big(\left( A,\alpha _{A}\right) \star \left( B,\alpha _{B}\right)\big) \times \big(\left( A,\alpha _{A}\right) \star \left( C,\alpha _{C}\right)\big) \to \left( A,\alpha _{A}\right) \star \left( B \times  C, \alpha _{B \times C}\right).
\]
 It is an easy task to verify that $\phi^{\star}$ and $\psi^{\star}$ are inverses to each other. This completes the proof.
 \end{proof}

\begin{Rem}
In the  case $\alpha_A = \Id\mathbf{}_A$, $\alpha_B = \Id_B$ and  $\alpha_C = \Id_C$, Proposition \ref{product} recovers the similar result for Lie algebras in \cite[Lemma 7]{El1}.
\end{Rem}

\begin{Ex} Examples of Hom-Lie algebras satisfying the requirements of Proposition \ref{product} are:
\begin{enumerate}
\item[a)] Metabelian Lie algebras $A = B = C$ together with $\alpha_A=\alpha_B=\alpha_C=0$ and the Hom-actions given by the structural bracket.
\item[b)] Hom-Lie algebras $(A, \alpha_A)$, $(B, \alpha_B)$, $(C,\alpha_C)$  together with trivial Hom-actions of $(A,  \alpha_A)$ and $(B, \alpha_B)$ on each other, and trivial Hom-actions of $(A,  \alpha_A)$ and $(C, \alpha_C)$ on each other.
\end{enumerate}
\end{Ex}

Let $\left( G,\alpha _{G}\right)$ and $\left( H,\alpha _{H}\right)$ be Hom-Lie algebras with compatible Hom-actions on each other, and let $\left( K,\alpha_{K}\right)$ be a Hom-ideal of both $\left( G,\alpha _{G}\right)$ and $\left( H,\alpha _{H}\right)$. For example, this is the case if $\left( G,\alpha _{G}\right)$, $\left( H,\alpha _{H}\right)$ and $\left( K,\alpha_{K}\right)$ all are Hom-ideals of a Hom-Lie algebra and $K\subseteq G\cap H$.
Let us consider the homomorphisms $\iota_1=\inc \star \Id: (K \star H, \alpha_{K \star H}) \rightarrow (G \star H, \alpha_{G \star H})$ and
$\iota_2=\Id \star \inc  :(G \star K, \alpha_{G \star K}) \rightarrow (G \star H, \alpha_{G \star H}))$ and introduce the following notations
  $K_1 = \im(\iota_1)$ and $K_2 = \im(\iota_2)$. Then both $(K_1, \alpha_{G \star H \mid})$ and $(K_2, \alpha_{G \star H \mid})$ are Hom-ideals of $\left( G \star H,\alpha _{G \star H} \right)$.

\begin{Pro}\label{quatient} There is an isomorphism of Hom-Lie algebras
\[
 \phi :\left( \frac{G}{K},  \overline{\alpha}_{G} \right)  \star \left( \frac{H}{K},\overline{ \alpha}_{H}\right) \overset{\cong}{\longrightarrow} \left( \frac{ G\star H}{K_{1} + K_{2}}, \overline{\alpha}_{G \star H} \right)
\]
defined by  $ \phi \left( (g+K) \star (h+K) \right) =  g\star h + \left( K_{1}+K_{2}\right)$,
where $\overline{\alpha}$ denotes the  endomorphism induced by $\alpha$ on the quotient.
\end{Pro}
\begin{proof}
It is readily checked that $\phi$ is a well-defined surjective homomorphism.
Now, take $\left( g + K\right) \star \left( h + K \right) \in \Ker(\phi)$, then  $g \star h \in K_1 + K_2$. So $g \star h$ can be rewritten as $\iota_1(k_1 \star h_1) + \iota_2 (g_2 \star k_2)$ and, consequently, $(g +K) \star (h +K ) = (k_1 + K) \star (h_1 + K) + (g_2 + K) \star (k_2 + K) = 0$, for $g, g_2 \in G$, $h, h_1 \in H$, $k_1, k_2 \in K$. This means that $\Ker(\phi)=0$ and the proof is  completed.
\end{proof}

\section{On nilpotency of the non-abelian tensor product}

In this section, we prove that the non-abelian tensor product of Hom-Lie algebras $\left( M\star N, \alpha_{M\star N} \right)$  inherits the property of nilpotency, solvability and Engel from the respective property of the Hom-Lie algebra $({^NM}, \alpha_{{M\mid}})$ (or $({^MN}, \alpha_{{^N\mid}})$). The investigations below are extensions of results for Lie algebras given in \cite{STME} to the case of Hom-Lie algebras.

  \begin{De}
We say that a Hom-Lie algebra $(L, \alpha_L)$ is $k$-Engel if $R^k_x=0$ for all $x \in L$, where
\[
R^k_x:L\rightarrow L, \quad y \mapsto\underset {k}{[\dots [[y, \underbrace{x],x], \dots, x]}}
\]
is the $k$-adjoint representation of $(L, \alpha_L)$, defined for $x\in L$.
 \end{De}

\begin{Th} \label{bounds} \
Let  $(M,\alpha_M)$ and $(N,\alpha_N)$ be Hom-Lie algebras with compatible Hom-actions on each other.
 \begin{enumerate}
 \item[a)] If $\left( {^NM},\alpha_{{M\mid}} \right)$  is nilpotent, then so are $\left(M \star N,\alpha_{M \star N} \right)$ and $\left( {^MN},\alpha_{_{N\mid}} \right)$. Furthermore, if  $\left( {^NM},\alpha_{{M\mid}} \right)$ is nilpotent of class $k$, ~then
    \[
     k \leq {\sf ncl} \left( M \star N,  \alpha_{M \star M} \right) \leq k+1 \quad \text{and} \quad {\sf ncl} \left( {^MN}, \alpha_{{^MN}} \right)\leq k+1.
     \]

     Here ${\sf ncl}(-)$ denotes the class of nilpotency.

  \item[b)] Let $(M,\alpha_M)$ and $(N, \alpha_N)$ be Hom-Lie algebras such that  $\alpha_M$ and $\alpha_N$ are surjective endomorphisms  or both $(M,\alpha_M)$ and $(N, \alpha_N)$ satisfy the weak $\alpha$-identity condition (see Remark \ref{conditions}). If
      $\left( {^NM},\alpha_{{M\mid}} \right)$  is solvable, then so are $\left( M \star N,\alpha_{M \star N}  \right)$ and $\left( {^MN},\alpha_{{N\mid}} \right)$. Furthermore, if  $\left( {^NM},\alpha_{{M\mid}} \right)$ is solvable of class $k$, then
      \[
      k \leq {\sf scl} \left( M \star N, \alpha_{M \star N} \right) \leq k+1 \quad \text{and} \quad {\sf scl} \left({^MN}, \alpha_{{N\mid}} \right)\leq k+1.
      \]
    Here ${\sf scl}(-)$ denotes the class of solvability.

  \item[c)] If $\left( {^NM},\alpha_{{M\mid}} \right)$ is $k$-Engel, then $\left( M \star N,\alpha_{M \star N} \right)$ and $\left( {^MN},\alpha_{{N\mid}} \right)$ are $(k+1)$-Engel.
 \end{enumerate}
\end{Th}
\begin{proof}
{\it a)} If $\left( {^NM},\alpha_{{M\mid}} \right)$  is nilpotent of class $k$, then by Theorem \ref{action-on-tensor} \emph{a)} we have the following isomorphisms
 \begin{equation}\label{4.1}
\frac{(M \star N,\alpha_{M \star N})}{\Ker(\psi_M)}\cong \left( {^NM},\alpha_{{M\mid}} \right) ~~   \text{ and} ~~
\frac{(M \star N,\alpha_{M \star N})}{\Ker(\psi_N)}\cong \left( {^MN},\alpha_{{N\mid}} \right).
 \end{equation}
By Theorem \ref{action-on-tensor} \emph{c)} the Hom-ideals $\Ker(\psi_M)$ and $\Ker(\psi_N)$ are in the center of $(M \star N,\alpha_{M \star N})$. Now
\[
\psi_{M}(M \star N, \alpha_{M \star N})^{[k+1]} = \psi_{M}\big[(M \star N, \alpha_{M \star N})^{[k]}, (M \star N, \alpha_{M \star N})\big] \subseteq \left( {^N}M \right)^{[k+1]} = 0.
 \]
So $(M \star N, \alpha_{M \star N})^{[k+1]} \subseteq \Ker(\psi_{M}) \subseteq \Z(M \star N, \alpha_{M \star N})$, then $(M \star N, \alpha_{M \star N})^{[k+2]} = 0$, that is $(M \star N, \alpha_{M \star N})$ is nilpotent Hom-Lie algebra with class of nilpotency at most $k+1$.

On the other hand, if $(M \star N, \alpha_{M \star N})^{[l+1]}=0$, by using the first isomorphism in (\ref{4.1}),  we get
\[
\left( ^{N}{M}, \alpha_{M\mid}\right)^{[l+1]} \cong \left(\frac{(M \star N, \alpha_{M \star N})}{\Ker(\psi_{M})}\right)^{[l+1]} \cong \frac{(M \star N, \alpha_{M \star N})^{[l+1]}}{\Ker(\psi_{M})} = 0.
\]
 Since $\left( ^{N}{M}, \alpha_{M \mid} \right)$ is nilpotent of class $k$, then $k \leq l$.

By the same reason, $\left( ^{M}{N}, \alpha_{N\mid}\right)^{[l+1]}=0$, when $(M \star N, \alpha_{M \star N})^{[l+1]}=0$, that is $\left( ^{M}{N}, \alpha_{N \mid} \right)$ is nilpotent with class of nilpotency $\leq l$, but $l \leq k+1$. Then the proof of {\it a)} is completed.

{\it b)} This is similar to the proof of {\it a)} and will be omitted.

{\it c)} Suppose that $({^N}M, \alpha_{M\mid})$ is  $k$-Engel, then for all $m \star n, m'\star n' \in M \star N$,
$R_{m \star n}^k(m' \star n') = \underset {k}{[\dots [[m'\star n', \underbrace{m\star n],m\star n], \dots, m\star n]}}\in \Ker(\psi_M) \subseteq \Z(M \star N)$. Indeed,
\[
\psi_M\! \underset {k}{\left(\! [\dots\! [[m'\star n', \underbrace{m\star n],m\star n], \!\dots\!, m\star n]}\! \right) }\!= \! \underset {k}{[\dots\! [[-\!^{n'}m', \underbrace{-\!^nm],-\!^nm], \!\dots\!, -\!^nm]}}= 0.
\]
Then $(M \star N)^{[k+1]}=0$ since
\[
\underset {k+1}{[\dots [[m'\star n', \underbrace{m\star n],m\star n], \dots, m\star n]}} = [R_{m \star n}^k(m' \star n'), m \star n]  =0.
\]
 So $(M \star N,\alpha_{M \star N})$ is $(k+1)$-Engel.

On the other hand,
\[
R_{{^mn}}^{k+1}(^{m'}{n'}) = [\dots[[ ^{m'}{n'}, \underbrace{{^mn}], {^mn}], \dots, {^mn}}_{k+1}] =  \psi_N\left(R_{{m \star n}}^{k+1}({m'}\star {n'}) \right) = 0,
\]
 i. e. $\left( {^MN},\alpha_{{N\mid}} \right)$ is $(k+1)$-Engel.
\end{proof}

\begin{Rem}
It is clear that Theorem \ref{bounds} still remains true when $({^NM},\alpha_{M\mid})$ and  $({^MN},\alpha_{N\mid})$ are replaced with each other.
\end{Rem}

\begin{Co} \label{bounds square}
Let  $(M,\alpha_M)$ be a Hom-Lie algebra. Then
\begin{enumerate}
\item[a)] If the derived Hom-subalgebra $([M,M],\alpha_{M\mid})$ is nilpotent of class $k$,  then $(M \star M, \alpha_{M \star M})$ is nilpotent of class $k$ or $k+1$.
\item[b)] Suppose, in addition, $\alpha_M$ is a surjective endomorphism, or  $(M,\alpha_M)$ satisfies the weak $\alpha$-identity condition (see Remark \ref{conditions}) . If the derived  Hom-subalgebra $([M,M],\alpha_{M\mid})$ is solvable of class $k$, then $(M \star M, \alpha_{M \star M})$ is solvable of class $k$ or $k+1$.
\item[c)] If the derived algebra subalgebra $([M,M],\alpha_{M\mid})$ is $k$-Engel,  then $(M \star M, \alpha_{M \star M})$ is $(k+1)$-Engel.
\end{enumerate}
\end{Co}
\begin{proof}
This is a particular case of Theorem \ref{bounds}.
\end{proof}
\bigskip

The following examples show that both outcomes obtained in the above results for the nilpotency and solvability class  of the tensor product can be occurred.

\begin{Ex}\
\begin{enumerate}
\item[a)]Let $(M,\alpha_M)$ be a nilpotent Hom-Lie algebra of class $1$, with the  Hom-action on itself given by the bracket in $M$. Then for $m_i \in M$, $1\leq i \leq 4$, we have
 \begin{equation*}
 [m_1 \star m_2, m_3 \star m_4]=-(^{m_2}m_1) \star (^{m_3}m_4)= - [m_2,m_1] \star [m_3, m_4]=0
  \end{equation*}
Hence $(M \star M)^{[1]} = 0$  implying that ${\sf ncl}(M \star M,\alpha_{M \star M})= {\sf ncl}([M,M], \alpha_{M\mid})$.

\item[b)]Let $(M,\alpha_M)$ be a Hom-Lie algebra such that  $\alpha_M$ is a surjective endomorphism or $(M,\alpha_M)$ satisfies the weak $\alpha$-identity condition (see Remark \ref{conditions}). If $(M,\alpha_M)$ is solvable of class $1$, with the Hom-action on itself given by the bracket in $M$, then
     $(M \star M)^{(1)} = 0$,  implying that ${\sf scl}(M \star M,\alpha_{M \star M})={\sf scl}([M,M], \alpha_{M\mid})$.

\item[c)]  Let $(M, \alpha_M)$ be the three-dimensional Hom-Lie algebra with basis $\{a_1,a_2,a_3\}$, endomorphism $\alpha_M=0$ and bracket operation given by $[a_1,a_2] = - [a_2,a_1] = a_2$ and $[a_2,a_3] = - [a_3,a_2] = a_2$.

    Consider Hom-ideals $(G, 0)$ and $(H, 0)$ of $(M, 0)$ spanned by $\{a_1,a_2\}$ and $\{a_2,a_3\}$, respectively. Of course, $(G, 0)$ and  $(H, 0)$ act compatibly on each other by the Hom-action given by the structural  bracket. A direct checking shows that $^{H}{G} = \langle \{a_2 \} \rangle$ is nilpotent of class 1; $(G \star H)^{[1]} = \langle \{a_2 \star a_2 \} \rangle$ is nilpotent of class  2 and $^{G}{H} = \langle \{a_2 \} \rangle$ is nilpotent of class 1, illustrating the bounds provided by Theorem \ref{bounds}.

    Note that, since $(M, 0)$ satisfies the weak $\alpha$-identity condition, the same conclusions about solvability are true.

\item[d)] Let $(M, \alpha_M)$ be the four-dimensional Hom-Lie algebra with basis $\{a_1,a_2,a_3,$ $a_4\}$, endomorphism $\alpha_M=0$ and bracket operation given by $[a_1,a_2] = - [a_2,a_1] = a_3$, $[a_1,a_3] = - [a_3,a_1] = a_4$ and $[a_2,a_3] = - [a_3,a_2] = a_1$.

    Then $([M,M], 0) = (\langle \{ a_1, a_3, a_4 \} \rangle, 0)$ is nilpotent of class 2 and $(M\star M, 0)$ is nilpotent of class 3, since $(M\star M)^{[1]}$ is spanned by $\{ a_1 \star a_1, a_1 \star a_3, a_1 \star a_4, a_3 \star a_1, a_3 \star a_3, a_3 \star a_4, a_4 \star a_1, a_4 \star a_3, a_4 \star a_4 \}$, illustrating the bounds provided by Corollary \ref{bounds square}.

       Since  $(M, 0)$ satisfies the weak $\alpha$-identity condition, the same conclusions about solvability are true.

\item[e)] Let $(M, \alpha_M)$ be the two-dimensional Hom-Lie algebra with basis $\{a_1,a_2\}$, bracket operation given by $[a_1,a_2]=-[a_2,a_1]= a_1$ and endomorphism $\alpha$ represented by the matrix $\left( \begin{array}{cc} 0 & 1 \\ 0 & 1 \end{array} \right)$ \cite{CaInPa}. Then $([M,M], \alpha_{M\mid})$ is 1-Engel and $(M \star M, \alpha_{M \star M})$ is 2-Engel, illustrating the bounds provided by Corollary \ref{bounds square}.
\end{enumerate}
\end{Ex}

\begin{De}
Let $(Q, \alpha_Q)$ and $(M, \alpha_M)$ be Hom-Lie algebras and there is a Hom-action of $(Q, \alpha_Q)$ on $(M, \alpha_M)$. We say that $(Q, \alpha_Q)$ acts nilpotently on $(M, \alpha_M)$ if $(M, \alpha_M)$ has a series
\[
0 = M_0 \unlhd M_1 \unlhd \dots \unlhd M_{k-1} \unlhd M_k = M
\]
such that ${^x}m_{i+1} \in M_i$ for all $x \in Q$, $m_{i+1} \in M_{i+1}$, $0\leq i < k$.
\end{De}

\begin{Ex}
Let $(M, \alpha_M)$ and $(N, \alpha_N)$ be Hom-ideals of a Hom-Lie algebra $(Q, \alpha_Q)$ such that $N \subset M$.  If $({M}/{N},\overline{\alpha}_{M})$ has a central series, then  $(M, \alpha_M)$  acts nilpotently on $({M}/{N},\overline{\alpha}_{M})$ by means of the Hom-action of $(M, \alpha_M)$ on $({M}/{N},\overline{\alpha}_{M})$ given by   $^{m'}(m + N) = [m',m] + N$, for all $m', m \in M$.
\end{Ex}

\begin{Th} \label{epimorphism}
For any Hom-Lie algebra $(Q, \alpha_{Q})$, the map
\[
\varphi : \left( \frac{Q^{[i]}}{Q^{[i+1]}}, \overline{\alpha}_{Q \mid} \right) \star \left( \frac{Q}{[Q,Q]}, \overline{\alpha}_{Q \mid} \right)  \longrightarrow \left( \frac{Q^{[i+1]}}{Q^{[i+2]}}, \overline{\alpha}_{Q \mid} \right), \quad i\geq 0,
\]
 given by
\begin{equation} \label{varphi}
\varphi\left( \left(x+Q^{[i+1]} \right) \star \left( y + [Q,Q] \right) \right)= [x,y]+ Q^{[i+2]}
 \end{equation}
 is a surjective  Hom-Lie algebra homomorphism.
\end{Th}
\begin{proof} First of all we observe that Hom-actions of $\left( \frac{Q^{[i]}}{Q^{[i+1]}}, \overline{\alpha}_{{Q}\mid} \right)$ and $\left( \frac{Q}{[Q,Q]},  \overline{\alpha}_{Q \mid} \right)$  on each other are induced by the bracket in $Q$ and clearly both are trivial and therefore compatible.
 On the other hand, the bilinear map
 \[
 h : \left( \frac{Q^{[i]}}{Q^{[i+1]}}, \overline{\alpha}_{Q \mid} \right) \times \left( \frac{Q}{[Q,Q]},  \overline{\alpha}_{Q \mid} \right) \to \left( \frac{Q^{[i+1]}}{Q^{[i+2]}}, \overline{\alpha}_{Q \mid} \right)
 \]
 given by $h \left( \left( x+ Q^{[i+1]} \right), \left( y + [Q, Q] \right) \right) = [x,y]+ Q^{[i+2]}$ is a well-defined Hom-Lie pairing, then by Lemma \ref{Lie pairing} $h$ induces the homomorphism $\varphi$, which clearly is a surjection.
\end{proof}

\begin{Co}
If $(Q, \alpha_Q)$ is a nilpotent Hom-Lie algebra such that $\left( \frac{Q}{[Q,Q]}, \overline{\alpha}_{Q \mid} \right)$ is finite-dimensional, then  $(Q, \alpha_Q)$ is finite-dimensional.
\end{Co}
\begin{proof}
This is a simple modification of the proof of \cite[Corollary 1]{EG} .
\end{proof}

\begin{Le} \label{action tensor}
Let $(G, \alpha_{Q \mid}), (H, \alpha_{Q \mid}), (M, \alpha_{Q \mid}), (N, \alpha_{Q \mid})$ be Hom-ideals of a Hom-Lie algebra $(Q, \alpha_Q)$ such that $H \subseteq G$, $N \subseteq M$ and there are compatible Hom-actions of $({G}/{H}, \overline{\alpha}_{Q \mid})$ and $({M}/{N}, \overline{\alpha}_{Q \mid})$ on each other.

If $(Q, \alpha_Q)$ acts nilpotently on $({G}/{H}, \overline{\alpha}_{Q \mid})$ and $({M}/{N}, \overline{\alpha}_{Q \mid})$, then $(Q, \alpha_Q)$ acts nilpotently on $({G}/{H}, \overline{\alpha}_{Q \mid})\star ({M}/{N}, \overline{\alpha}_{Q \mid})$ as well.
\end{Le}
\begin{proof}
Since $(Q, \alpha_Q)$ acts nilpotently on $({G}/{H}, \overline{\alpha}_{Q \mid})$ and $({M}/{N}, \overline{\alpha}_{Q \mid})$, then there are series
$$0 = \frac{G_0}{H}\unlhd \frac{G_1}{H} \unlhd \dots \unlhd \frac{G_t}{H} = \frac{G}{H}$$
and
$$0 = \frac{M_0}{N}\unlhd \frac{M_1}{N} \unlhd \dots \unlhd \frac{M_s}{N} = \frac{M}{N}$$
such that ${^q}(g_{i+1}+H) \in  {G_i}/{H}$ and ${^q}(m_{j+1}+N) \in  {M_j}/{N}$,  for all $q\in Q$, $g_{i+1} \in {G_{i+1}}$, $m_{j+1} \in{M_{j+1}}$,  $0 \leq i < t$, $0 \leq j < s$.

To make formulas shorter, we denote $(T, \alpha_T)=({G}/{H}, \overline{\alpha}_{Q \mid})\star ({M}/{N}, \overline{\alpha}_{Q \mid})$.
Straightforward calculations show  that there is a Hom-action of $(Q, \alpha_Q)$ on $(T, \alpha_T)$ given by
\begin{equation} \label{action}
{^q}\big((g+ H) \star (m + N)\big) = \big([q,g]+H\big) \star \big({\alpha}_{M}(m)+ N\big) + \big({\alpha}_{G}(g)+H\big) \star \big([q,m]+N\big)
\end{equation}

 It is routine to check that the following series of Hom-ideals of $(T, {\alpha}_{T})$:
\[
0 = (T_0, {\alpha}_{T \mid})  \unlhd (T_1,  {\alpha}_{T \mid}) \unlhd \dots \unlhd (T_{t+s},  {\alpha}_{T \mid}) = (T,  {\alpha_{T}}),
\]
where $T_r\!= \!\big\langle \{ ({g}_i\!+\!H) \star ({m}_j\!+\!N) \mid g_i\in G_i,  m_j\in M_j, i+j \leq r,   0 \leq i \leq t,  0 \leq j \leq s \} \big\rangle$, $0\leq r\leq t+s$,
is the one confirming that $(Q, \alpha_Q)$ acts nilpotently on $(T, \alpha_T)$  by means of the  Hom-action defined in (\ref{action}).
\end{proof}

\begin{Th} \label{3.11}
Let $(M, \alpha_{M})$ be a Hom-ideal of a Hom-Lie algebra $(Q, \alpha_Q)$ with surjective endomorphism $\alpha_Q$. If $(M, \alpha_{M})$ and $\left( \frac{Q}{[M,M]}, \overline{\alpha}_{Q} \right)$ are nilpotent Hom-Lie algebras, then $(Q, \alpha_{Q})$ is nilpotent as well.
\end{Th}
\begin{proof}  Since $\left( \frac{Q}{[M,M]}, \overline{\alpha}_{Q} \right)$ is nilpotent, then by Theorem \ref{nilpotent} there exists a central series as follows:
\begin{equation} \label{central series}
 0 = \frac{Q_0}{[M,M]} \unlhd \frac{Q_1}{[M,M]} \unlhd \dots \unlhd \frac{Q_t}{[M,M]} =  \frac{Q}{[M,M]}.
\end{equation}
By taking $M_j = Q_j \cap M$, $0 \leq j \leq t$,  we construct the series
\begin{equation*}
0 =  \frac{M_0}{[M,M]} \unlhd \frac{M_1}{[M,M]} \unlhd \dots \unlhd \frac{M_t}{[M,M]} =  \frac{M}{[M,M]}.
\end{equation*}
Then it is easy to see that $(Q, \alpha_Q)$ acts nilpotently on $\left( \frac{M}{[M,M]}, \overline{\alpha}_{M} \right)$ by means of the Hom-action
\begin{equation} \label{action nilpot}
{^q} (m + [M, M]) = [q,m] + [M, M], ~~ q \in Q, m \in M.
\end{equation}
For $i \geq 0$, let $F_i = \frac{M^{[i]}}{M^{[i+1]}}$. We claim that $(Q, \alpha_Q)$ acts nilpotently on $(F_i, \overline{\alpha}_{M \mid})$. This can be shown by induction on $i$.  Namely, for $i=0$, $(Q, \alpha_Q)$ acts nilpotently on $(F_0, \overline{\alpha}_{M \mid})$ by means of the Hom-action given in
 (\ref{action nilpot}). Now, if we assume that  $(Q, \alpha_Q)$ acts nilpotently on $(F_i, \overline{\alpha}_{M \mid})$, then by Lemma \ref{action tensor} we have that  $(Q, \alpha_Q)$ acts nilpotently on $(F_i, \overline{\alpha}_{M \mid}) \star \left( \frac{M}{[M, M]}, \overline{\alpha}_{M} \right)$. By Theorem \ref{epimorphism}, $(F_{i+1}, \overline{\alpha}_{M \mid})$ is an image of $(F_i, \overline{\alpha}_{M \mid}) \star \left( \frac{M}{[M,M]}, \overline{\alpha}_{M} \right)$ and hence $(Q, \alpha_Q)$ acts nilpotently on $(F_{i+1}, \overline{\alpha}_{M \mid})$ as well.

Since $(M, \alpha_{M})$ is nilpotent, then there exists $k \in \mathbb{Z}^+$ such that $M^{[k]} \neq 0$ and $M^{[k+1]} = 0$. Then by combining (\ref{central series}) with the lower central series of $(M, \alpha_{M})$, we get the following central series
\[
0 = M^{[k+1]} \unlhd M^{[k]} \unlhd \dots \unlhd M^{[1]} = [M, M] = Q_0 \unlhd Q_1 \unlhd \dots \unlhd Q_t = Q,
\]
 which means that $(Q, \alpha_{Q})$ is nilpotent.
 \end{proof}

\begin{Rem}
 In the case $\alpha_Q = \Id_Q$, Theorem \ref{3.11}
recovers the corresponding result for Lie algebras given in \cite[Theorem 2]{EG}.
\end{Rem}

The following example shows that the requirement of surjectivity for  the endomorphism $\alpha_Q$ in Theorem \ref{3.11} is really needed in order to establish the result.
\begin{Ex}
 Let $(Q, \alpha_Q)$ be the three-dimensional Hom-Lie algebra with basis $\{a_1,a_2,a_3\}$, endomorphism $\alpha_Q =0$ (which is not surjective) and bracket operation given by $[a_1, a_2] = - [a_2, a_1]= a_2$; $[a_2, a_3] = - [a_3, a_2]= a_3$. Then $(M = \langle \{ a_2, a_3 \} \rangle, 0)$ is a nilpotent Hom-ideal of class $2$; $(Q/[M, M], 0) \cong (\langle \{ a_1, a_2 \} \rangle, 0)$ is nilpotent of class $2$. However $(Q, \alpha_Q)$ is not nilpotent.
\end{Ex}


\section{The case of perfect Hom-Lie algebras}

In this section we establish conditions under which the non-abelian tensor product of Hom-Lie algebras  preserves perfect objects. Moreover, we describe compatibility of the non-abelian tensor product with the universal central extensions of two perfect Hom-Lie algebras.

First of all let us recall that a Hom-Lie algebra $(M,\alpha_M)$ is \emph{perfect} if $M=[M,M]$.

Note that the non-abelian tensor product of prefect Hom-Lie algebras is not perfect, in general. For instance, let $(M,\alpha_M)$ and $(N, \alpha_N)$ be perfect Hom-Lie algebras with trivial (so compatible) Hom-actions on each other. Nevertheless $(M \star N, \alpha_{M \star N})$ is not perfect since $[M \star N, M \star N]=0$. But the non-abelian tensor product may be perfect if additional requirements are demanded.

\begin{Pro} \label{perfect}
Let $(M, \alpha_M)$ and $(N, \alpha_N)$ be perfect Hom-Lie algebras, with compatible Hom-action on each other, such that ${^N}M=M$ and  ${^M}N=N$. Then $(M \star N, \alpha_{M \star N})$ is a perfect Hom-Lie algebra as well.
\end{Pro}
\begin{proof}
The obvious inclusions $M \star N \subseteq {^N}M \star {^M}N \subseteq [ M \star N, M \star N]$ prove the result.
\end{proof}

\begin{Co}\label{corollary6.2}
If $(M, \alpha_M)$  is a perfect Hom-Lie algebra, then $(M \star M, \alpha_{M \star M})$ is a perfect Hom-Lie algebra as well.
\end{Co}

In the following example we give of Hom-Lie algebras satisfying the requirements of Proposition \ref{perfect}.

\begin{Ex} \label{example}
 Let $(M,\alpha_M)$ be the perfect three-dimensional Hom-Lie algebra with basis $\{e_1, e_2, e_3 \}$, $\alpha_M=0$, and bracket operation given by
\begin{align*}
[e_1,e_2]=-[e_2,e_1]= e_3, & &  [e_1,e_3]=-[e_3,e_1]= e_2, & & [e_2,e_3]= - [e_3, e_2]=e_1.
\end{align*}
Let $(N,\alpha_N)$ be the perfect four-dimensional Hom-Lie algebra with basis $\{a_1, a_2, a_3, a_4 \}$, $\alpha_N=0$, and bracket operation given by
 \begin{align*}
[a_1,a_2]=-[a_2,a_1]= a_3, & &  [a_1,a_3]=-[a_3,a_1]= a_2, & & [a_1,a_4]=- [a_4,a_1]=a_4,\\
 [a_2,a_3]= - [a_3, a_2]=a_1. && &&
\end{align*}
   The Hom-action of $(M,\alpha_M)$ on $(N,\alpha_N)$ given by
   \begin{align*}
^{e_1}a_2 = -^{e_2}a_1= a_3, &  ^{e_1}a_3 = -^{e_3}a_1= a_2,& ^{e_1}{a_4}= - ^{e_4}{a_1} =  a_4, &  ^{e_2}e_3 = - ^{e_3}a_2 = a_1.
\end{align*}
   and the Hom-action of $(N,\alpha_N)$ on $(M,\alpha_M)$ given by
   \begin{align*}
^{a_1}e_2 = -^{a_2}e_1= e_3, & &  ^{a_1}e_3 = -^{a_3}e_1= e_2, & & ^{a_2}e_3 = - ^{a_3}e_2 = e_1.
\end{align*}
are compatible Hom-actions. Moreover $^NM=M$ and $^MN=N$. Keep in mind that the unwritten brackets and actions are trivial.
\end{Ex}

A \emph{central extension} of a Hom-Lie algebra $(L, \alpha_L)$ is an epimorphism of Hom-Lie algebras $(K,\alpha_K) \stackrel{\pi}\twoheadrightarrow (M, \alpha_M)$ such that $[\Ker(\pi), K] = 0 $, i. e. $\Ker(\pi) \subseteq \Z(K)$.
It is called \emph{a universal central extension}  if in addition, for every central  extension $ (K',\alpha_{K'}) \stackrel{\pi'} \twoheadrightarrow (M, \alpha_M) $ there exists a unique homomorphism of Hom-Lie algebras $h : (K,\alpha_K) \to (K',\alpha_{K'})$ such that $\pi' \circ h = \pi$.

It is known by \cite[Theorem 4.11]{CaInPa} that a Hom-Lie algebra $(M, \alpha_M)$ admits a universal central extension if and only if $(M, \alpha_M)$ is perfect. The construction of the  universal central extension is as follows:  let  $\frak{uce}(M)$ be the quotient vector space $M \wedge M/\im(d_3)$, where $d_3$ denotes the boundary of the Hom-Lie homology complex (see \cite{CaInPa} for details). Let $\{x, y \}$ denote the coset of $x\wedge y\in M \wedge M$  in $\frak{uce}(M)$. The bracket operation  $[\{x_1,x_2\},\{y_1,y_2\}]= \{[x_1,x_2],[y_1,y_2]\}$ on $\frak{uce}(M)$ together with the naturally induced endomorphism $\alpha_{\frak{uce}(M)}:\frak{uce}(M)\rightarrow \frak{uce}(M)$, provides Hom-Lie algebra $(\frak{uce}(M), \alpha_{\frak{uce}(M)})$.
  Then
   \[
  (H_2^{\alpha}(M), \alpha_{\frak{uce}(M)\mid})\rightarrowtail (\frak{uce}(M), \alpha_{\frak{uce}(M)}) \overset{u}\twoheadrightarrow (M, \alpha_M), \quad u\{m_1, m_2\}=[m_1,m_2],
  \]
    is the universal central extension of the perfect Hom-Lie algebra $(M, \alpha_M)$, where $H_2^{\alpha}(M)$ denotes the second homology with trivial coefficients of $(M, \alpha_M)$.

There is another description of the universal central extension in \cite[Theorem 4.4]{CaKhPa} via the non-abelian tensor product. Namely, given a perfect Hom-Lie algebra  $(M,\alpha_M)$, then
\[
(M\star M, \alpha_{M \star M})\overset{\psi_M}\twoheadrightarrow (M, \alpha_M),  \quad \psi_M(m \star m')=[m,m'],
\]
 is the universal central extension of $(M,\alpha_M)$. Since the universal central extension is unique up to isomorphisms, we have
 \[
  (\frak{uce}(M), \alpha_{\frak{uce}(M)})\cong (M\star M, \alpha_{M \star M}), \quad  \{x,y \} \mapsto x \star y.
 \]

 In what follows we will use the notation  $(\frak{uce}(M), \alpha_{\frak{uce}(M)})$ even for $(M\star M, \alpha_{M \star M})$. Thus, by Corollary \ref{corollary6.2} if $(M, \alpha_M)$ is a perfect Hom-Lie algebra, then $(\frak{uce}(M), \alpha_{\frak{uce}(M)})$ is perfect as well.

\begin{Pro}
For a  perfect Hom-Lie algebra $(G, \alpha_G)$ with surjective endomorphism $\alpha_G$, there is an exact sequence of Hom-Lie algebras
\begin{equation*}
(H,\alpha_H)\overset{\eta} \longrightarrow (\frak{uce}(G) \star \frak{uce}(G), \alpha) \overset{\psi_G \ast \psi_G} \longrightarrow (G \star G, \alpha_{G \star G}) \longrightarrow 0 \ ,
\end{equation*}
where $H=\big(H_2^{\alpha}(G) \star \frak{uce}(G)\big) \rtimes \big(\frak{uce}(G) \star H_2^{\alpha}(G)\big)$ and $\alpha = \alpha_{\frak{uce}(G)} \star \alpha_{\frak{uce}(G)}$.
\end{Pro}
\begin{proof}
First note that compatible Hom-actions of $(G, \alpha_G)$ on itself and of $(\frak{uce}(G), \alpha_{\frak{uce}(G)})$ on itself are given by the corresponding structural bracket.
 Since $\alpha_G$ is  surjective, then $\alpha_{\frak{uce}(G)}$ is surjective as well. The homomorphism $\psi_G : (\frak{uce}(G), \alpha_{\frak{uce}(G)}) \twoheadrightarrow (G, \alpha_G)$ preserves Hom-actions. Indeed,
\begin{align*}
\psi_G \left( ^{g_1 \star g_2}(g_1' \star g_2') \right) &= \psi_G \left[g_1 \star g_2, g_1' \star g_2'\right] = \psi_G \big( [g_1, g_2] \star  [g_1', g_2'] \big)  \\
&= \left[ [g_1, g_2],  [g_1', g_2'] \right] = \left[ \psi_G(g_1 \star g_2), \psi_G(g_1' \star g_2') \right] = ^{\psi_G(g_1 \star g_2)} \psi_G(g_1' \star g_2'). \end{align*}
Similarly, it can be checked that $\alpha_{\frak{uce}(G)}$ also preserves the Hom-actions. Then the assertion follows by applying Proposition \ref{exact-tensor-3} for the short exact sequence
\[
0\longrightarrow (H_2^{\alpha}(G), \alpha_{\frak{uce}(G)\mid})\longrightarrow(\frak{uce}(G), \alpha_{\frak{uce}(G)}) \overset{\psi_G }\longrightarrow (G, \alpha_G)\longrightarrow 0 \ .
\]
\end{proof}

\begin{Co}
For a  perfect Hom-Lie algebra $(G, \alpha_G)$ with surjective endomorphism $\alpha_G$, there is a central extension
\begin{equation} \label{exact seq}
0 \longrightarrow \eta(H)  \longrightarrow (\frak{uce}(G) \star \frak{uce}(G), \alpha) \overset{\psi_G \ast \psi_G} \longrightarrow (G \star G, \alpha_{G \star G}) \longrightarrow 0 \ .
\end{equation}
\end{Co}

\begin{Th} \label{Th 5.5}
For any perfect Hom-Lie algebra $(G, \alpha_G)$ with surjective endomorphism $\alpha_G$, there is a surjective homomorphism of Hom-Lie algebras
\[
\omega :  (\frak{uce}(G \star G), \alpha_{\frak{uce}{(G \star G)}}) \twoheadrightarrow (\frak{uce}(G) \star \frak{uce}(G), \alpha) \ .
\]
\end{Th}
{\begin{proof}
Since $(G, \alpha_G)$ is perfect, then so is $(G \star G, \alpha_{G \star G})$ by Proposition \ref{perfect}. Hence it admits the following universal central extension
(see \cite[Theorem 4.11 {\it c)} and {\it d)}]{CaInPa})
\[
0\longrightarrow (H^{\alpha}_2(G \star G), \alpha_{\frak{uce}(G \star G)\mid})\longrightarrow (\frak{uce}(G \star G), \alpha_{\frak{uce}(G \star G)})\overset{\psi_{G \ast G}}\longrightarrow (G \star G, \alpha_{G \star G}) \longrightarrow 0.
\]

Now consider the central extension (\ref{exact seq}), then there is a unique  homomorphism of Hom-Lie algebras $\omega:(\frak{uce}(G \star G), \alpha_{\frak{uce}(G \star G)})\rightarrow (\frak{uce}(G) \star \frak{uce}(G), \alpha)$,  making commutative the following diagram:
\[
\xymatrix{
0\ar[r] & (H^{\alpha}_2(G \star G), \alpha_{\frak{uce}(G \star G)\mid})\ar[r] \ar@{-->}[d]^{\omega_\mid} & (\frak{uce}(G \star G), \alpha_{\frak{uce}(G \star G)})\ar[r]^{~~~~\psi_{G \ast G}} \ar@{-->}[d]^{\omega} & (G \star G, \alpha_{G \star G}) \ar[r] \ar@{=}[d] & 0 \\
0 \ar[r] & \eta(H)  \ar[r] & (\frak{uce}(G) \star \frak{uce}(G), \alpha) \ar[r]^{~~~~\psi_G \ast \psi_G}  & (G \star G, \alpha_{G \star G}) \ar[r] & 0
}
\]
It is easy to show that $\omega$ is  a surjective homomorphism.
\end{proof}

\begin{Co}\label{cor 6.7}
Let $(G, \alpha_G)$ be a perfect Hom-Lie algebra with surjective endomorphism $\alpha_G$ and $H_2^{\alpha}(G \star G)= \im(\eta)$, then $$(\frak{uce}(G \star G), \alpha_{\frak{uce}(G \star G)}) \cong (\frak{uce}(G) \star \frak{uce}(G), \alpha).$$
\end{Co}
\begin{proof}
The surjective homomorphism $\omega$ in Theorem \ref{Th 5.5} is also injective under this additional condition.
\end{proof}


\section*{Acknowledgements}

The first and second authors were supported by  Agencia Estatal de Investigación (Spain), grant MTM2016-79661-P (AEI/FEDER, UE, support included).


\end{document}